\documentclass[11pt]{article}
\usepackage{amsmath,amssymb,amsfonts,amsthm}
\usepackage{float}
\numberwithin{equation}{section}
\newtheorem{theorem}{Theorem}[section]
\newtheorem{definition}{Definition}[section]
\newtheorem{lemma}[theorem]{Lemma}

\begin{document}
\begin{center}
{\Large{\textbf{Contemplating on Brush Numbers of Mycielski Jaco Graphs, $\mu(J_n(1)), n \in \Bbb N$}}} 
\end{center}
\vspace{0.5cm}
\large{\centerline{(Johan Kok, Susanth C, Sunny Joseph Kalayathankal)\footnote {\textbf {Affiliation of author:}\\
\noindent Johan Kok (Tshwane Metropolitan Police Department), City of Tshwane, Republic of South Africa\\
e-mail: kokkiek2@tshwane.gov.za\\ \\
\noindent Susanth C (Department of Mathematics, Vidya Academy of Science and Technology), Thalakkottukara, Thrissur-680501, Republic of India\\
e-mail: susanth\_c@yahoo.com\\ \\
\noindent Sunny Joseph Kalayathankal (Department of Mathematics, Kuriakose Elias College), Mannanam, Kottayam- 686561, Kerala, Republic of India\\
e-mail: sunnyjoseph2014@yahoo.com}}
\vspace{0.5cm}
\begin{abstract}
\noindent The concept of the brush number $b_r(G)$ was introduced for a simple connected undirected graph $G$. The concept will be applied to the Mycielski Jaco graph $\mu(J_n(1)), n \in \Bbb N,$  in respect of an \emph{optimal orientation} of $J_n(1)$ associated with $b_r(J_n(1))$.\\ \\
\noindent Further to the above the concept of a \emph{brush centre} of a simple connected graph will be introduced. Because brushes themselves may be technology of kind, the technology in real worl application will normally be the subject of maitenance or calibration or virus vetting or alike. Finding a \emph{brush centre} of a graph will allow for well located maintenance centres of the brushes prior to a next cycle of cleaning.
\end{abstract}
\noindent {\footnotesize \textbf{AMS Classification Numbers:} 05C07, 05C12, 05C20, 05C38, 05C70} 
\section{Introduction}
\noindent  For a general reference to notation and concepts of graph theory see [1]. For ease of self-containess we shall briefly introduce the concepts of \emph{brush numbers} and \emph{Mycielskian graphs}.
\subsection{The brush number of a simple connected graph $G$}
\noindent The concept of the brush number $b_r(G)$ of a simple connected graph $G$ was introduced by McKeil [5] and Messinger et. al. [7]. The problem is initially set that all edges of a simple connected undirected graph $G$ is \emph{dirty}. A finite number of brushes, $\beta_G(v) \geq 0$ is allocated to each vertex $v \in V(G).$ Sequentially any vertex which has $\beta_G(v) \geq d(v)$ brushes allocated may clean the vertex $v$ and send exactly one brush along a dirty edge and in doing so allocate an additional brush to the corresponding adjavent vertex (neighbour). The reduced graph $G' = G - vu_{\forall vu \in E(G), \beta_G(v) \geq d(v)}$ is considered for the next iterative cleaning step. Note that a neighbour of vertex $v$ in $G$ say vertex $u$, now have $\beta_{G'}(u) =\beta_G(u) + 1.$\\ \\
\noindent Clearly for any simple connected undirected graph $G$ the first step of cleaning can begin if and only if at least one vertex $v$ is allocated, $\beta_G(v)= d(v)$ brushes. The minimum number of brushes that is required to allow the first step of cleaning to begin is, $\beta_G(u) = d(u) = \delta(G).$ Note that these conditions do not guarantee that the graph will be cleaned. The conditions merely assure at least the first step of cleaning.\\ \\
\noindent If a simple connected graph $G$ is orientated to become a directed graph, brushes may only clean along an out-arc from a vertex. Cleaning may initiate from a vertex $v$\ if and only if $\beta_G(v) \geq d^+(v)$ and $d^-(v) =0.$ The order in which vertices sequentially initiate cleaning is called the \emph{cleaning sequence} in respect of the orientation $\alpha_i$. The minimum number of brushes to be allocated to clean a graph for a given orientation $\alpha_i(G)$ is denoted $b_r^{\alpha_i}$. If an orientation $\alpha_i$ renders cleaning of the graph undoable we define $b_r^{\alpha_i} = \infty.$ An orientation $\alpha_i$ for which $b_r^{\alpha_i}$ is a minimum over all possible orientations is called \emph{optimal}.\\ \\
Now, since the graph $G$ having $\epsilon(G)$ edges can have $2^{\epsilon(G)}$ orientations, the \emph{optimal orientation} is not necessary unique. Let the set $\Bbb A =\{\alpha_i|\emph{ $\alpha_i$ an orientation of G}\}.$
\begin{lemma}
For a simple connected directed graph $G$, we have that:\\ \\ $b_r(G) = min_{\emph{over all $\alpha_i \in \Bbb A$}}(\sum_{v \in V(G)}max\{0, d^+(v) - d^-(v)\}) = min_{\forall \alpha_i}b_r^{\alpha_i}.$
\end{lemma}
\begin{proof}
See [9].
\end{proof}
\noindent Although we mainly deal with simple connected graphs it is easy to see that for set of simple connected graphs $\{G_1, G_2, G_3, ..., G_n\}$ we have that, $b_r(\cup_{\forall i}G_i) = \sum\limits_{i=1}^{n}b_r(G_i).$
\subsection{Mycielskian graph $\mu(G)$ of a graph, $G$}
\noindent Mycielski [8]  introduced an interesting graph transformation in 1955. The transformation can be described as follows:\\ \\
(1) Consider any simple connected graph $G$ on $n \geq 2$ vertices labelled $v_1, v_2, v_3, ..., v_n$.\\
(2) Consider the extended vertex set $V(G) \cup \{x_1, x_2, x_3, ..., x_n\}$ and add the edges $\{v_ix_j, v_jx_i|$ iff $v_iv_j \in E(G)\}.$\\
(3) Add one more vertex $w$ together with the edges $\{wx_i| \forall i\}.$\\ \\
The transformed graph \emph{(Mycielskian graph of G or Mycielski G)} denoted $\mu(G),$ is the simple connected graph with $V(\mu(G)) = V(G) \cup \{x_1, x_2, x_3, ..., x_n\} \cup \{w\}$ and $E(\mu(G)) = E(G) \cup \{v_ix_j, v_jx_i|$ iff $v_iv_j \in E(G)\} \cup \{wx_i| \forall i\}.$ 
\section{Brush Numbers of Mycielski Jaco Graphs, $\mu(J_n(1)), n \in \Bbb N$}
The infinite Jaco graph (\emph{order 1}) was introduced in $[2],$ and defined by $V(J_\infty(1)) = \{v_i| i \in \Bbb N\}$, $E(J_\infty(1)) \subseteq \{(v_i, v_j)| i, j \in \Bbb N, i< j\}$ and $(v_i,v_ j) \in E(J_\infty(1))$ if and only if $2i - d^-(v_i) \geq j.$\\ \\ The graph has four fundamental properties which are; $V(J_\infty(1)) = \{v_i|i \in \Bbb N\}$ and, if $v_j$ is the head of an edge (arc) then the tail is always a vertex $v_i, i<j$ and, if $v_k,$ for smallest $k \in \Bbb N$ is a tail vertex then all vertices $v_ \ell, k< \ell<j$ are tails of arcs to $v_j$ and finally, the degree of vertex $k$ is $d(v_k) = k.$ The family of finite directed graphs are those limited to $n \in \Bbb N$ vertices by lobbing off all vertices (and edges arcing to vertices) $v_t, t > n.$ Hence, trivially we have $d(v_i) \leq i$ for $i \in \Bbb N.$\\ \\
For ease of reference we repeat a few definitions found in [2].
\begin{definition}
The infinite Jaco Graph $J_\infty(1)$ is defined by $V(J_\infty(1)) = \{v_i| i \in \Bbb N\}$, $E(J_\infty(1)) \subseteq \{(v_i, v_j)| i, j \in \Bbb N, i< j\}$ and $(v_i,v_ j) \in E(J_\infty(1))$ if and only if $2i - d^-(v_i) \geq j.$
\end{definition}
\begin{definition}
The family of finite Jaco Graphs are defined by $\{J_n(1) \subseteq J_\infty(1)|n\in \Bbb {N}\}.$ A member of the family is referred to as the Jaco Graph, $J_n(1).$
\end{definition}
\begin{definition}
The set of vertices attaining degree $\Delta (J_n(1))$ is called the Jaconian vertices of the Jaco Graph $J_n(1),$ and denoted, $\Bbb{J}(J_n(1))$ or, $\Bbb{J}_n(1)$ for brevity.
\end{definition}
\begin{definition}
The lowest numbered (indiced) Jaconian vertex is called the prime Jaconian vertex of a Jaco Graph.
\end{definition}
\begin{definition}
If $v_i$ is the prime Jaconian vertex of a Jaco Graph $J_n(1)$, the complete subgraph on vertices $v_{i+1}, v_{i+2}, \cdots,v_n$ is called the Hope subgraph of a Jaco Graph and denoted,  $\Bbb{H}(J_n(1))$ or, $\Bbb{H}_n(1)$ for brevity.
\end{definition}
\noindent  Note that the Fisher Algorithm determines $d^+(v_i)$ on the assumption that the Jaco Graph is always sufficiently large, so at least $J_n(1), n \geq i+ d^+(v_i).$ So technically the value $d^+(v_i)$ determined by the Fisher Algorithm is, $d^+(v_i) = d_{J_{\infty}}^+(v_i).$ For a smaller graph the degree of vertex $v_i$ is given by $d_{J_n(1)}(v_i) = d^-(v_i) + (n-i).$ In $[2]$ Bettina's theorem describes an arguably, closed formula to determine $d^+(v_i)$. Since $d^-(v_i) = n - d^+(v_i)$ it is then easy to determine $d_{J_n(1)}(v_i)$ in a smaller graph $J_n(1), n< i + d^+(v_i).$ It is important to note that Definition 2.2 read together with Definition 2.1, prescribes a well-defined orientation of the underlying Jaco graph. So we have one defined orientation of the $2^{\epsilon(J_n(1))}$ possible orientations. In [4] the following theorem is proven.
\begin{theorem}
For the finite Jaco Graph $J_n(1), n \in \Bbb N,$ with prime Jaconian vertex $v_i$ we have that:\\ \\
$b_r(J_n(1)) = \sum\limits_{j=1}^{i}(d^+(v_j) - d^-(v_j)) + \sum\limits_{j = i+1}^{n}max\{0, (n-j)- d^-(v_j)\}.$
\end{theorem}
\noindent In general we have that if $\beta_{G'}(v)$ at a particular cleaning step has $\beta_{G'}(v) > d_{G'}(v)$, exactly $\beta_{G'}(v) - d_{G'}(v)$ brushes are left redundant and can clean along new edges linked to vertex $v$ if such are added through transformation of the graph $G$. The latter observation allows for an adaption of Theorem 2.1 to obtain the \emph{brush number} of $\mu(J_n(1)) n \geq 3$. Note that $\mu(J_1(1)) \simeq K_1 \cup P_2,$ hence a disconnected graph. Easy to see that $\mu(J_2(1)) \simeq C_5$ hence $b_r(\mu(J_2(1)) = 2.$
\begin{theorem}
For the Jaco graph $J_n(1), n \in \Bbb N, n \geq 2$ the brush number of the Mycielski Jaco graph is given by:\\ \\
$b_r(\mu(J_n(1)) = 2\sum\limits_{i=1}^{n}d^+_{J_n(1)}(v_i).$
\end{theorem}
\begin{proof}
Consider the Jaco graph, $J_3(1)$. From [4] it follows that $b_r(J_3(1)) =1$ with brush allocations $\beta_{J_3(1)}(v_1) =1, \beta_{J_3(1)}(v_2) = 0, \beta_{J_3(1)} = 0.$ In the graph $\mu(J_3(1))$ we add the set of vertices $\{x_1, x_2, x_3\} \cup \{w\}$ and we add the edges $\{v_1x_2, v_2x_1, v_2x_3, v_3x_2\} \cup \{wx_1, wx_2, wx_3\}.$\\ \\
Clearly $v_1$ has degree, $d_{\mu(J_3(1))}(v_1) =2.$ So besides the normal cleaning sequence within $J_3(1)$, vertex $v_1$ must either dispatch a second brush along edge $v_1x_2$ or await a brush dispatched from vertex $x_2$. In the latter case a brush will be left redundant at vertex $v_1$. So without loss of generality and to ensure optimality, allocate a second brush to $v_1$. At this stage we have the \emph{first brush number term} of $\mu(J_3(1)),$ $2d^+_{J_3(1)}(v_1) = 2.$\\ \\
On initiating the cleaning process one brush will be dispatched to $v_2$ along the arc $(v_1, v_2)$ (edge $v_1v_2$). In the next step of cleaning the brush from vertex $v_1$ now at $v_2$ can be dispatched to $v_3$ provided that two more brushes are initially allocated to $v_2$ to begin with. But in $\mu(J_3(1))$ we have the additional edges $v_2x_1, v_2x_3,$ so two additional brushes must be allocated to vertex $v_2$. So we have the \emph{second brush number term} of $\mu(J_3(1))$, $2d^+_{J_3(1)}(v_2) = 2.$\\ \\
In the second step of cleaning the brush, initially dispatched from $v_1$ to $v_2$ can be dispatched to vertex $v_3$. The third step of cleaning may proceed with no further allocation to $v_3$. Hence $v_3$ requires \emph{zero} additional brushes. Now we have the \emph{third and final brush number term} of $\mu(J_3(1))$, $2d^+_{J_3(1)}(v_3) = 0.$ Hence we have the result:\\ \\
$b_r(\mu(J_3(1))) = 2\sum\limits_{i=2}^{3}d^+_{J_3(1)}(v_i).$\\ \\
We settled the result through induction. Assume the the result holds for $J_k(1)$, with prime Joconian vertex $v_j$. Thus we assume that:\\ \\
$b_r(\mu(J_k(1))) = 2\sum\limits_{i=2}^{k}d^+_{J_k(1)}(v_i),$ holds.\\ \\
Now consider the Jaco graph $J_{k+1}(1)$ to begin with. This extension adds the vertex $v_{k+1}$ and the set of arcs, $\{(v_{i+1},v_{k+1}), (v_{i+2},v_{k+1}), (v_{i+3},v_{k+1}), ..., (v_n,v_{k+1})\}$ to $J_k(1)$ to obtain $J_{k+1}(1).$\\ \\
Each vertex $v_j, i+1 \leq j \leq k$ receives two additional arcs namely $(v_j,v_{k+1})$ and $(v_j,x_{k+1})$ in the transformed Mycielski Jaco graph. The minimum additional brushes per such vertex is thus two. We have that the respective brush number terms of $\mu(J_{k+1}(1))$ are $(d^+_{J_k}(v_j) + 2) = 2d^+_{J_{k+1}(1)}(v_j).$ This implies we have the partial brush number:\\ \\
$2\sum\limits_{i=1}^{k} d^+_{J_{k+1}(1)}(v_i).$\\ \\
With regards to the vertex $v_{k+1}$ we have, $d_{J_{k+1}(1)}(v_{k+1}) = d^-(v_{k+1})$ stemming from the $k-j$ vertices, $v_{j+1}, v_{j+2}, v_{j+3}, ..., v_k.$ Hence, $k-i$ brushes will be allocated to vertex $v_{k+1}$ through the iterative cleaning process. In the Mycielski Jaco graph the additional edges $v_{k+1}x_k, v_{k+1}x_{k-1}, v_{k+1}x_{k-2}, ..., v_{k+1}x_{j+1}$, (\emph{exactly $k-j$ edges}) have been added to vertex $v_{k+1}$ so no additional brushes are needed. It means that the \emph{final brush number term} is, $2d^+_{J_{k+1}(1)}(v_{k+1}) = 2.0 = 0.$ Since all the brush number terms of $\mu(J_{k+1}(1))$ have now been determined at the absolute minimum we have the result, $b_r(\mu(J_{k+1}(1)) = 2\sum\limits_{i=1}^{k+1}d^+(v_i).$\\ \\
Through induction we settle the result that:\\ \\
$b_r(\mu(J_n(1)) = 2\sum\limits_{i=1}^{n}d^+_{J_n(1)}(v_i), \forall J_n(1), n \in \Bbb N, n \geq 3.$
\end{proof}
\section{Brush Centre of a Graph}
\noindent From Theorem 2.1 and Lemma 1.1 the brush allocations can easily be determined for Jaco graphs. See [4]. For example, $J_9(1)$ requires the minimum brush allocations, $\beta_{J_9(1)}(v_1) = 1, \beta_{J_9(1)}(v_2) = 0, \beta_{J_9(1)}(v_3) = 1, \beta_{J_9(1)}(v_4) = 2, \beta_{J_9(1)}(v_5) = 1, \beta_{J_9(1)}(v_6) = 1, \beta_{J_9(1)}(v_7) = 0, \beta_{J_9(1)}(v_8) = 0, \beta_{J_9(1)}(v_9) = 0.$ We note that the allocations of $\beta(v_i) > 0$ are located at vertices $v_1, v_3, v_4, v_5, v_6$. The \emph{end} allocation itself is a minimum allocation associated with an optimal orientation.\\ \\
So far cleaning was restricted to a brush transversing a dirty edge only once. If the latter restriction is relaxed to, after the first complete cleaning sequence a brush may transverse an edge for a second time for another complete reversed cleaning sequence, the initial allocation of brushes or a deviation thereof can be obtained.This observation leads to the concept of a \emph{brush centre}. The question is, what is the minimium set of vertices, $\Bbb B_r(G) \subseteq V(G)$ (\emph{primary condition}) to allocate the $b_r(G)$ brushes to, to ensure cleaning of graph $G$ and on return \emph{(second cleaning)} the brushes are clustered as centrally as possible for \emph{maintenance} (\emph{secondary condition is the min(max(distance between vertices of the brush centre))}).  Finding a \emph{brush centre} of a graph will allow for well located maintenance centres of the brushes prior to a next cycle of cleaning. Because brushes themselves may be technology of kind, the technology in real worl application will normally be the subject of maitenance or calibration or virus vetting or alike.\\ \\
It is easy to see that for the path $P_n$ the allocation of one brush to either $\{v_1\}$ or $\{v_n\}$ is a minimum set and clustered absolutely centrally so both represent a \emph{brush centre}. So $P_n$ has two possible \emph{brush centres}. Similarly, the allocation of two brushes to any set $\Bbb B_r(C_n) = \{v_i\}, v_i \in V(C_n)$ of the cycle cycle $C_n$ represents a \emph{brush centre}. So $C_n$ has $n$ possible \emph{brush centres}. Hence the \emph{brush centre} of a graph $G$ is not necessary unique. However, for the star $K_{1,n}$ the \emph{brush centre} is indeed unique, namely, $\Bbb B_r(K_{1,n}) = \{v_1\}_{\emph{($v_1$ central)}}$, with $\beta_{K_{1,n}}(v_1)_{\emph{($v_1$ central)}} = n.$
\subsection{Brush centre of the Mycielski Jaco Graph, $\mu(J_n(1)), n \in \Bbb N$}
Let us immediately jump paths and consider $J_5(1)$. In the defined Jaco graph $J_5(1)$ the brush number is $b_r(J_5(1)) = 2$ [4], with the brush allocation $\beta_{J_5(1)}(v_1) =1, \beta_{J_5(1)}(v_2) = 0, \beta_{J_5(1)}(v_3) =1, \beta_{J_5(1)}(v_4) = 0, \beta_{J_5(1)}(v_5) = 0.$ We note that after the first cleaning sequence both brushes are allocated to the vertex $v_5$. The latter allocation of brushes with an appropriate re-orientation of $J_5(1)$ also clean the Jaco graph. On a second cleaning sequence the brushes can \emph{park} at $ v_5$ for \emph{maintenance}. Clearly the set $\{v_5\}$ with $\beta_{J_5(1)}(v_5) =2$ is a (\emph{the}) brush centre.
\begin{theorem}
Consider the initial minimal brush allocation of $b_r(J_n(1))$ brushes to the finite Jaco graph, $J_n(1), n \in \Bbb N$, [4]. The location of the brushes at the end of the cleaning sequence represents a brush centre of $J_n(1), n \in \Bbb N.$ 
\end{theorem}
\begin{proof}
For the paths $J_1(1), J_2(1), J_3(1), J_4(1)$ the result is obvious. As observed in the introduction the set $\Bbb B_r(J_5(1)) =\{v_5\}$ with $\beta_{J_5(1)}(v_5) =2$ is a (\emph{the}) brush centre of $J_5(1)$.\\ \\
We prove the result through induction. Assume the result holds for $J_k(1)$ which has the prime Jaconian vertex $v_i$. So our assumption implies that after the first cleaning sequence the brushes are clustered amongst vertices in the vertex set $\Bbb B_r \subseteq \{v_{i+1},, v_{i+2}, v_{i+3}, .., v_k\}$ in compliance with both the \emph{primary condition} and the \emph{secondary condition} namely, the min(max(distance between vertices of the brush centre)). Note that $min(max(d_{v,u \in \Bbb B_r}(v,u))) = 1$ because $\Bbb B_r(J_k(1)) \subseteq \Bbb H(J_k(1)).$\\ \\
Now consider the Jaco graph $J_{k+1}(1)$ to begin with. This extension adds the vertex $v_{k+1}$ and the set of edges, $\{v_{j+1}v_{k+1}, v_{j+2}v_{k+1}, v_{j+3}v_{k+1}, ..., v_nv_{k+1}\}$ to $J_k(1)$ to obtain $J_{k+1}(1).$\\ \\
Each vertex $v_j, i+1 \leq j \leq k$ receives one additional edges namely $v_jv_{k+1}.$ In addition the prime Jaconian vertex may, or may not change to $v_{i+1}$. Exactly $k-i$ new edges were added in the extension from $J_k(1))$ to $J_{k+1}(1)$. Since $b_r(J_{k+1}(1)) \geq b_r(J_k(1))$ additional brushes might be needed at some of the vertices $v_{i+1}, v_{i+2}, ..., v_{k_1}$ if $v_i$ remains the prime Jaconian vertex of $J_{k+1}(1).$ Else, additional brushes might be needed at some of the vertices $v_{i+2}, ..., v_{k_1}$. Note that the vertex $v_k$ will not require additional brushes. Since the minimum additional brushes to be allocated is always possible (\emph{primary condition}) and $\Bbb B_r(J_{k+1}(1)) \subseteq \Bbb H(J_{k+1})$ the $min(max(d_{v,u \in \Bbb B_r(J_{k+1}(1))}(v,u))) = 1$ (\emph{secondary condition}), the result follows for $J_{k+1}(1).$\\ \\
Hence the result is settled for all Jaco graphs, $J_n(1), n \in \Bbb N.$
\end{proof}
\textbf{\emph{Open access:}} This paper is distributed under the terms of the Creative Commons Attribution License which permits any use, distribution and reproduction in any medium, provided the original author(s) and the source are credited. \\ \\
References (Limited) \\ \\
$[1]$  Bondy, J.A., Murty, U.S.R., \emph {Graph Theory with Applications,} Macmillan Press, London, (1976). \\
$[2]$ Kok, J., Fisher, P., Wilkens, B., Mabula, M., Mukungunugwa, V., \emph{Characteristics of Finite Jaco Graphs, $J_n(1), n \in \Bbb N$}, arXiv: 1404.0484v1 [math.CO], 2 April 2014. \\
$[3]$  Kok, J., Fisher, P., Wilkens, B., Mabula, M., Mukungunugwa, V., \emph{Characteristics of Jaco Graphs, $J_\infty(a), a \in \Bbb N$}, arXiv: 1404.1714v1 [math.CO], 7 April 2014. \\ 
$[4]$ Kok, J., \emph{A note on the Brush Number of Jaco Graphs, $J_n(1), n \in \Bbb N$}, arXiv: 1412.5733v1 [math.CO], 18 December 2014. \\
$[5]$ McKeil, S., \emph{Chip firing cleaning process}, M.Sc. Thesis, Dalhousie University, (2007).\\ 
$[6]$ Messinger, M. E., \emph{Methods of decontaminating a network}, Ph.D. Thesis, Dalhousie University, (2008).\\
$[7]$ Messinger, M.E., Nowakowski, R.J., Pralat, P., \emph{Cleaning a network with brushes}. Theoretical Computer Science, Vol 399, (2008), 191-205.\\
$[8]$ Mycielski, J., \emph{Sur le coloriages des graphes}, Colloq. Maths. Vol 3, (1955), 161-162.\\
$[9]$ Ta Sheng Tan,\emph{The Brush Number of the Two-Dimensional Torus,} arXiv: 1012.4634v1 [math.CO], 21 December 2010.
\end{document}